\documentclass[a4paper,10pt]{amsart}
\usepackage[bf,wide,cthms]{gewoehn} 

\begin{document}

\title{On a discrete version of length metrics}
\author{Pedro Z\"{u}hlke}
\maketitle

\begin{abstract}
	Let $ (X,d) $ be a metric space. We study a metric $ d_0 $ on $ X $
	naturally derived from $ d $. If $ (X,d) $ is complete and locally
	compact, or if it is complete and $ (d_0)_0=d_0 $, then
	$ d_0 $ coincides with the length metric induced by $ d $. 
	Counterexamples are constructed when any of
	the hypotheses is absent. The behavior of the
	iterates of $ d_0 $ (the metrics $ d_0^n $ recursively defined as $
	(d_0^{n-1})_0 $) is also considered. 
\end{abstract}


\setcounter{section}{-1}
\section{Introduction}\label{S:basic}

Most of the definitions and many of the unproved facts stated below can be 
found in both \cite{BriHae} and \cite{BurBurIva},
and the notation used here largely coincides with that of \cite{BurBurIva}. 

Let $ (X,d) $ be a metric space. Let $ P=\se{0=t_0<t_1<\dots<t_n=1} $ denote a
partition of $ [0,1] $ and $ \abs{P}=\max_k(t_k-t_{k-1}) $ its norm. 
The length $ L(\ga) $ of a (continuous) path $
\ga\colon [0,1]\to X $ is defined by
\begin{equation}\label{E:length}
	L(\ga)=\sup_{P}\sum_{k=1}^nd(\ga(t_{k-1}),\ga(t_{k}))=\lim_{\abs{P}\to
	0}\sum_{k=1}^nd(\ga(t_{k-1}),\ga(t_{k})),
\end{equation}
where the sup is taken over all finite partitions of $ [0,1] $ and the
second equality is easily proved. The \tdef{length metric} $ \bar{d} \colon
X\times X\to [0,\infty]$
\tdef{induced by $ d $} is given by 
\begin{equation*}
	\bar{d}(p,q)=\inf\set{L(\ga)}{\ga\colon [0,1]\to X\text{ is a  
		path in $ (X,d) $ joining $ p $ to $ q $}}\quad (p,q\in X).
\end{equation*}
Thus $ \bar{d}(p,q)=\infty $ if and only if no $ d $-rectifiable path 
connecting $
p $ and $ q $ exists. 

For $ n\in \N^+ $, the set $ \se{1,\dots,n} $ will be denoted by 
$ [n] $. Given $ \eps>0 $, an \tdef{$ \eps $-chain} joining $ p $
to $ q $ is a string $ c=(x_0,x_1,\dots,x_n) $ of points in $ X $ satisfying
$x_0=p $, $ x_n = q$ and $
d(x_{k-1},x_{k})\leq \eps $ for all $ k\in [n] $.  Its \tdef{length} $ L(c) $
equals $ \sum_{k=1}^nd(x_{k-1},x_k) $. For $ p,q\in X $, define 
\begin{equation*}
	d_{\eps}(p,q)=\inf\set{L(c)}{c \text{ is an $ \eps $-chain joining $ p $ to
	$ q $}}.
\end{equation*}
It is easily seen that $ d_\eps\colon X\times X\to [0,\infty]$ satisfies all the
axioms for a metric for any $ \eps>0 $. Notice that $ d_\eps(p,q)=d(p,q) $ if
the latter is not greater than $ \eps $. Finally, $ d_0\colon X\times X\to
[0,\infty] $ is defined by 
\begin{equation*}
	d_0(p,q) = \sup_{\eps>0}d_\eps(p,q) = \lim_{\eps\to 0}d_\eps(p,q)
	\quad (p,q\in X), 
\end{equation*}
where the second equality follows from the fact that $ d_{\eps'}\geq d_{\eps} $
if $ \eps'\leq \eps $. The proof that $ d_0 $ is indeed a metric is
left to the reader. It may be regarded as a discrete or discontinuous version
of the length metric $ \bar{d} $. The former is defined as the supremum of 
infima, and the latter as the infimum of suprema.  The purpose of this note
is to study $ d_0 $, especially in its relation to $ d $ and $
\bar{d} $. 

\subsection*{Summary of results} The main result states that $ d_0 $
agrees with $ \bar {d}$ provided that one of the following conditions is
satisfied: (i) $ (X,d) $ is complete and locally compact; (ii) $ (X,d) $ is
complete and $ (d_0)_0 =d_0	$; (iii) $ (X,d) $ is a length space (that is, 
$ \bar{d}=d $). Moreover, if (i) holds then any $p,q\in  X $ such that $
d_0(p,q)<\infty $ can be joined by a path whose $ d_0 $-length equals $ d_0(p,q)
$. This theorem is proved in \S\ref{S:general}. Examples showing that none of
its hypotheses can be omitted are constructed in \S\ref{S:examples}. In
particular, we exhibit a space $ (X,d) $ for which $ d_0 \neq \bar{d} $ even
though $ (X,d) $ is complete, $ \sig $-compact (hence separable), path-connected
and locally path-connected, both through rectifiable paths.

Given metrics $ \rho,\,\rho' $ on $ X $, let us write $ \rho\leq \rho' $ to mean
that $ \rho(p,q)\leq \rho'(p,q) $ for all $ p,q\in X $, and $ \rho<\rho' $ when
moreover strict inequality holds for at least one pair of points.  In
\S\ref{S:iterates} we study the iterates of $ d_0 $, i.e., the metrics $ d_0^n $
inductively defined by $ d_0^{n+1}=(d_0^n)_0 $, where $ d_0^0=d $.  It is proved
that if $ (X,d) $ is complete and $ d_0^n=d_0^{n+1} $ for some $ n\in \N $, then
$ d_0^{m}=\bar{d} $ for all $ m\geq n $.  We construct complete spaces $ (Y_n,d)
$ where 
\begin{equation*}
	d<d_0<\dots<d_0^n<d_0^{n+1}=\bar{d},
\end{equation*}
and a complete space $ (Y_\infty,d) $ for which $ d_0^n<d_0^{n+1}  $ for all $
n\in \N $, but $ \lim_nd_0^n<\bar{d} $. This should be compared to the
behavior of the $ \bar{\phantom{x}} $ operation, which is always idempotent.


\section{General results}\label{S:general}

\begin{lem}\label{L:ineqs}
	Let $ (X,d) $ be any metric space. Then 
\begin{equation}\label{E:ineqs}
	d\leq d_0\leq \bar{d}.
\end{equation}
\end{lem}
\begin{proof}
	It is clear that $ d\leq d_0 $ since  $ d\leq d_\eps $ for all $ \eps>0 $.
	Given a path $ \ga $ joining $ p $ to $ q $, a sum as in \eqref{E:length}
	coincides with the length of a corresponding $ \eps $-chain $
	(\ga(0),\ga(t_1),\dots,\ga(1)) $ as soon as $ \abs{P}\leq \eps
	$. Thus $ d_\eps(p,q)\leq L(\ga) $ for any such path $ \ga $, whence $
	d_\eps\leq \bar{d} $ for all $ \eps>0$ and $ d_0\leq \bar{d} $.
\end{proof}
\begin{lem}\label{L:completeness}
	If $ (X,d) $ is complete, then so are $ (X,d_0) $ and $ (X,\bar{d}) $. 
\end{lem}

\begin{proof}
	Suppose first that $ (x_n)_{n\in \N} $ is a $ d_0 $-Cauchy sequence. Then it
	is also $ d $-Cauchy by \eqref{E:ineqs}, hence it $ d $-converges to
	some $ x\in X $ by hypothesis. Let $ \eps>0 $ be given. Take $ n_0\in \N
	$ such that $ i,j\geq n_0 $ implies $ d_0(x_i,x_j)\leq \eps $. Given
	$ \de > 0 $, choose 
	$ m\geq n_0 $ so that $d(x_m,x)\leq \min\se{\de,\eps} $.
	Then $ d_\de(x_m,x)=d(x_m,x) $, hence
	\begin{equation*}
		d_\de(x_n,x) \leq d_\de(x_n,x_m)+d_\de(x_m,x)\leq
		d_0(x_n,x_m)+d(x_m,x)\leq 2\eps\text{ for all $ n\geq n_0 $}.
	\end{equation*}
	Since $\de$ is arbitrary, it follows that $d_0(x_n,x)\leq 2\eps$ for all
	$n\geq n_0$. Therefore $(x_n)$ $d_0$-converges to $x$.

	Suppose now that $ (x_n)_{n\in \N} $ is a $ \bar{d} $-Cauchy sequence. Then,
	as above, it must have a $ d $-limit $ x $.
	Choose 
	$ n_1\leq n_2\leq \dots $ such that $ \bar{d}(x_i,x_j)<2^{-\nu} $ for all 
	$i,j\geq n_\nu$. Let $ \eps >0 $ be given, take $\nu_0\in \N$
	satisfying  $2^{-\nu_0+1}\leq\eps$ and let $ n\geq n_{\nu_0} $. 
	Choose a path $\ga_0\colon
	[0,\frac{1}{2}]\to X$ joining $x_{n}$ to $x_{n_{\nu_0+1}}$ of length less
	than $ 2^{-\nu_0} $. For each $k\geq 1 $, let $ \ga_k\colon
	[1-2^{-k},1-2^{-k-1}] \to X $ be a path of length less than $ 2^{-\nu_0-k} $
	joining $ x_{n_{\nu_0+k}} $ to $ x_{n_{\nu_0+k+1}} $.
	Finally, define $ \ga\colon [0,1]\to X $ to be the concatenation 
	of all the $ \ga_k $. To be precise, set
	\begin{equation*}
		\ga(t)=\ga_k(t)\text{ if $ t\in	\big[1-2^{-k},1-2^{-k-1}\big] $\,\ 
		($ k\geq 0 $) \,and \,$ \ga(1)=x $.}
	\end{equation*}
	Since $(x_n)$ 
	$ d $-converges to $ x $, $ \ga $ is indeed $ d $-continuous at 
	$ t=1 $. Further, its length is at most $ 2^{-\nu_0+1} \leq \eps $. 
	Therefore $ \bar{d}(x_n,x)\leq \eps $ for all $n\geq n_{\nu_0}$ and $ (x_n)
	$ $ \bar{d} $-converges to $ x $.
\end{proof}

\begin{exm}\label{E:localcompactness}
	Let $ (X,d) $ be the subspace of $ \R^2 $ (with the Euclidean metric)
	which consists of $ [0,1]\times \se{0} $ and 
	the vertical segments of length 1 based at $ (0,0) $ and $
	(\frac{1}{k},0) $ for each $k\in \N^+$. Then $ (X,d) $ is compact and $
	d_0 = \bar{d} $, but $ (X,d_0) $ is not locally compact.
\end{exm}

\begin{rem}\label{R:ineqs2}
	Let $ \rho\leq \rho' $ be two metrics on $ X $. Then $ \bar{\rho}\leq
	\bar{\rho}' $ and $ \rho_0\leq \rho_0' $. Indeed, if $ \ga\colon [0,1]\to X
	$ is $ \rho' $-continuous, then it is $ \rho $-continuous and $
	L_{\rho}(\ga)\leq L_{\rho'}(\ga) $. Similarly, an $ \eps $-chain for $ \rho'
	$ is also an $ \eps $-chain for $ \rho $, and its $ \rho $-length is smaller
	than its $ \rho' $-length. 
\end{rem}

\begin{lem}\label{L:semicontinuity}
	Let $ (X,d) $ be a metric space. Then $ d_0 $ is lower semicontinuous with
	respect to $ d $:
	\begin{equation*}\label{E:semicontinuity}
		\liminf_{p_n\lto{d} p\,,\, q_n\lto{d\,} q}d_0(p_n,q_n)\geq d_0(p,q) \text{
		\ for all $ p,q \in X $.}
	\end{equation*}
\end{lem}
\begin{proof}
	For any $ p,q\in X $ and $ \eps>0 $,
	\begin{equation*}\label{E:semicontinuity1}
		\liminf_{p_n\lto{d} p\,,\, q_n\lto{d\,} q}d_0(p_n,q_n)\geq 
		\liminf_{p_n\lto{d} p\,,\, q_n\lto{d\,} q}d_{\eps}(p_n,q_n)= 
		d_\eps(p,q), 
	\end{equation*}
	where the equality comes from the fact that $ d_\eps $ agrees with $ d $
	at distances smaller than $ \eps $. Letting $ \eps\to 0 $ 
	the desired inequality is obtained. 
\end{proof}
\begin{urem}
	The length metric $ \bar{d} $ is generally not semicontinuous from either
	side with respect to $ d $, as can be shown by means of
	simple examples. Although $ d_\eps $ is continuous with respect to $ d $ for
	all $ \eps>0 $, $ d_0 $ itself need not be continuous, 
	as illustrated by $ X=\se{0}\cup\set{\frac{1}{k}}{k\in \N^+}
	\subs \R$. 
\end{urem}

\begin{thm}\label{T:main}
	Let $ (X,d) $ be a metric space. Suppose that one of the following holds:
	\begin{enumerate}
		\item [(i)] $ (X,d) $ is complete and locally compact.
		\item [(ii)] $ (X,d) $ is complete and $ (d_0)_0 = d_0 $. 
		\item [(iii)] $ (X,d) $ is a length space \tup(that is, 
			$ \bar{d} = d $\tup).
	\end{enumerate}
	Then $ d_0=\bar d $. Moreover, if \tup{(i)} holds then $ (X,d_0) $ is a
	geodesic space.
\end{thm}
The latter assertion means that for any $p,q\in X$ for which 
$d_0(p,q)<\infty$, there exists a $ d_0 $-continuous path $\ga\colon [0,1]\to X$
joining $ p $ to $ q $ whose $ d_0 $-length is $ d_0(p,q) $. 
Note that condition (i) is satisfied if $(X,d)$ is
proper (i.e., if any $d$-ball is precompact) and in particular if $ (X,d) $ is
compact.  In cases (ii) and (iii), it cannot
be guaranteed that $ (X,d_0) = (X,\bar{d)} $ is a geodesic space; e.g., let $ X
$ be the metric graph consisting of two vertices and one edge of length $
1+\frac{1}{k} $ connecting them for each $ k\in \N^+ $.

\begin{cor}\label{C:mixed}
		Let $ (X,d) $ be a metric space satisfying any of conditions
		\tup{(i)--(iii)}. Then 
		\begin{equation*}
			d_0=\bar{d}=(d_0)_0=\overline{d_0}=\bar{d}_0.
		\end{equation*}
\end{cor}
\begin{proof}
	Immediate from \eqref{E:ineqs}, \tref{T:main} and the fact that $
	\bar{\bar{d}} $ always coincides with $ \bar{d} $ 
	(a proof of the latter can be found in \cite{BriHae}, pp.~32--33
	or \cite{BurBurIva}, pp.~37--38).
\end{proof}

The remainder of this section is dedicated to the proof of the theorem.

\begin{lem}\label{L:smallerandsmaller}
	Let $ c_n $ be $ \de_n $-chains joining $ x,y\in (X,d) $, with $ \de_n\to 0
	$ as $ n\to\infty $ \tup($ n\in \N $\tup). Then
	\begin{equation*}
		d_0(x,y)\leq \liminf L(c_n).
	\end{equation*}
\end{lem}
\begin{proof}
Immediate from the relations
\begin{equation*}
	d_0(x,y)=\lim_nd_{\de_n}(x,y)\text{\ \,and\,\ }d_{\de_n}(x,y)\leq L(c_n).\qedhere
\end{equation*}
\end{proof}

\begin{lem}\label{L:technical}
	Let $ (X,d) $ be a metric space and $p,q\in X$,  $d_0(p,q)<\infty$.  For
	each $ n\in \N $, let $c_n=(x_0,\dots,x_{N_n})$ be an $\eps_n$-chain
	joining $p$ to $q$, with $ L(c_n)\to d_0(p,q) $ and $ \eps_n \to 0 $ as $
	n\to \infty $.  Assume the existence of $ k_n\in [N_n] $ so that
	$(x_{k_n})_{n\in \N}$ $d$-converges to some $x\in X$.  Then $
	d_0(p,q)=d_0(p,x)+d_0(x,q)$, 
	\begin{equation}\label{E:additivity}
		\sum_{k=1}^{k_n}d(x_{k-1},x_{k}) \to d_0(p,x)\text{\ \
		and\ \ }
		\sum_{k=k_n+1}^{N_n}d(x_{k-1},x_{k})\to d_0(x,q)\text{ \,as $ n\to\infty
		$.}
	\end{equation}
\end{lem}
\begin{proof}
	Set $\de_n=\max\se{\eps_n,d(x_{k_n},x)}$. Then for each $n\in \N$,
	$(x_0,\dots,x_{k_n},x)$ and $ (x,x_{k_{n}},\dots,x_{N_n}) $ are
	$\de_n$-chains joining $p$ to $x$ and $x$ to $ q $, respectively.
	Furthermore, if
	\begin{equation*}
		s_n=\sum_{k=1}^{k_n}d(x_{k-1},x_k)+d(x_{k_n},x)\text{\ \ and\ \ }
		t_n=d(x,x_{k_n}) + \sum_{k=k_n+1}^{N_n}d(x_{k-1},x_k),
	\end{equation*}
	then $(s_n+t_n)\to d_0(p,q)$ by hypothesis. By \lref{L:smallerandsmaller},
	$d_0(p,x)\leq \liminf s_n$ and $d_0(x,q)\leq \liminf t_n$. On the other
	hand,
	\begin{equation*}
		\limsup s_n + \liminf t_n\leq \lim(s_n+t_n) = d_0(p,q)\leq
		d_0(p,x)+d_0(x,q),
	\end{equation*}
	whence $ \limsup s_n\leq d_0(p,x) $. Thus $\lim s_n$ exists and equals
	$d_0(p,x)$; similarly, $\lim t_n=d_0(x,q)$. Consequently  
	$d_0(p,q)=d_0(p,x)+d_0(x,q)$.
\end{proof}

In what follows the open ball centered at $ p $ of radius $ r $ with
respect to a metric $ \rho $ is denoted by $ B_\rho(p;r) $. If $ S\subs X $, we
denote by $ B_\rho(S;r) $ the union of all balls $ B_\rho(p;r) $ with $ p\in S
$. Also, $ \lfloor t \rfloor $ denotes the greatest integer smaller than
or equal to $ t \in \R $. The main step in the proof of \tref{T:main} is the
following weak additivity property for $ d_0 $.

\begin{lem}\label{L:chain}
	Let $(X,d)$ be a locally compact metric space. Suppose that 
	$ B_{d_0}(p;r) $ is $d$-precompact and $r\leq  d_0(p,q) <\infty$. 
	Then for all sufficiently small $\de>0$, there exist 
	$p_0=p,\dots,p_N\in X$ \tup($N=\lfloor{r/\de}\rfloor$\tup) such
	that:
	\begin{equation}\label{E:minimizing}
		d_0(p,q)=d_0(p_0,p_1)+\dots+d_0(p_{N-1},p_N)+d_0(p_N,q)\text{\ and\ }
		d_0(p_{k-1},p_k)=\de\text{\ for all $k\in [N]$.}
	\end{equation}
\end{lem}

\begin{proof}
	Using local compactness of $(X,d)$, choose for each $y\in \ol{B_{d_0}(p;r)}$ 
	an $\eps_y>0$ such that $ B_d(y;2\eps_y) $ is $ d $-precompact. Now 
	extract a finite subcover of $ \ol{B_{d_0}(p;r)} $ by finitely many
	$B_d(y_j;\eps_{y_j})$ and set $ \eps=\min\{\eps_{y_j}\} $. It is easily
	checked that $B_{d}(x;\eps)$ is $ d $-precompact for any $x\in B_{d_0}(p;r)$.

	Let $\de\in (0,\eps)$.  For each $n\in \N^+$,
	choose a $\frac{1}{n}$-chain $(x_0,\dots,x_{N_n})$ connecting $p$ to $q$ of
	length less than $ d_{\frac{1}{n}}(p,q)+\frac{1}{n} $,
	and let $k_n$ be the greatest element of $[N_n] $ satisfying 
	\begin{equation*}
		\sum_{k=1}^{k_n}d(x_{k-1},x_k)< \de.
	\end{equation*}
	Then $x_{k_n}\in B_d(p;\de)$ for all $n$ by the triangle inequality. 
	Since this ball is $ d $-precompact, passing to a
	subsequence if necessary, it can be assumed that 
	$(x_{k_n})$ $d$-converges to some $p_1\in X$ as $n\to \infty$. From  
	\lref{L:technical} it follows that $d_0(p,q)=d_0(p,p_1)+d_0(p_1,q)$ and 
	\begin{equation*}
		d_0(p,p_1)=\lim_{n\to\infty}\sum_{k=1}^{k_n}d(x_{k-1},x_{k})
		=\de
	\end{equation*}
	by the choice of $k_n$. To obtain $p_2$,
	apply the same procedure with $p_1$ in place of $p$, and so on inductively.
\end{proof}

\begin{cor}\label{C:balldistance}
	Let $(X,d)$ be a locally compact metric space. Suppose that $ B_{d_0}(p;r) $
	is $d$-precompact and $ d_0(p,q) \geq r $. Then
	\begin{equation*}
	d(q,B_{d_0}(p;r))\leq d_0(q,B_{d_0}(p;r))=d_0(p,q)-r.
	\end{equation*}
\end{cor}
\begin{proof}
	Immediate from \lref{L:chain} and the inequality $d\leq d_0$.
\end{proof}

\begin{lem}\label{L:precompact}
	If $(X,d)$ is locally compact and complete, then any $ d_0 $-ball
	is $ d $-precompact. 
\end{lem}
\begin{proof}
	Let $ p\in X $ be arbitrary. By local compactness of $(X,d)$, there exists
	some $\de>0$ such that $B_d(p;\de)$ is $ d $-precompact. Since
	$B_{d_0}(p;\de)\subs B_{d}(p;\de)$, we have that 
	\begin{equation*}
		R:=\sup\set{r\in \R}{B_{d_0}(p;r)\text{ is $ d $-precompact}}\geq \de.
	\end{equation*}
	Suppose for the sake of obtaining a contradiction that
	$ R $ is finite. 
	Let $\eps\in (0,R)$ be arbitrary. Using the fact that
	$B_{d_0}(p;R-\eps)$ is $ d $-totally bounded, cover the latter by 
	finitely many balls $B_d(p_j;\eps)$ ($ j\in [m]) $, 
	where $p_j\in B_{d_0}(p;R-\eps)$ for
	each $j$. Then, by \cref{C:balldistance},
	\begin{equation*}
		B_{d_0}(p;R)\subs \bcup_{j=1}^m B_d(p_j;2\eps).
	\end{equation*}
	Therefore $B_{d_0}(p;R)$ is $ d $-totally bounded. Its $ d $-closure is also
	totally bounded and in addition complete, as a closed subset of the complete
	space $(X,d)$. Thus $B_{d_0}(p;R)$ is $ d $-precompact. In particular,  it 
	can be covered by finitely many balls $ B_d(x_i;\de_{i}) $ ($ i\in [l] $)
	such that each $ B_d(x_i;2\de_i) $ is $ d $-precompact. Again by
	\cref{C:balldistance}, if $ \de=\min_i\se{\de_i} $ then	
	\begin{equation*}
		B_{d_0}(p;R+\de) \subs \bcup_{i=1}^lB_d(x_i;2\de_i).
	\end{equation*}
	Hence the former is $ d $-precompact, contradicting the choice of $ R $. 
\end{proof}

\begin{urem}
	It need not be true that the $ d_0 $-balls are $ d_0 $-precompact even if $
	(X,d) $ is compact, as shown by \eref{E:localcompactness}.
\end{urem}

\begin{proof}[Proof of \tref{T:main}]
	Suppose that hypothesis (i) of the theorem holds and let $p,q\in X$ be 
	arbitrary. If $ d_0(p,q) $ is not finite, then by
	\eqref{E:ineqs} neither is $ \bar d(p,q)$, hence in this
	case they coincide. Assume then that $ r=d_0(p,q) $ is finite.
	The ball $B_{d_0}(p;r)$ is precompact by \lref{L:precompact}. Applying 
	\lref{L:chain} one deduces that
	for all sufficiently large $ n\in \N $, say $n\geq n_0$, it is possible 
	to find $ p=p^n_0,\dots,p^n_{N_n+1}=q $ for which 
	\eqref{E:minimizing} holds with $\de=\frac{1}{n}$, where $ N_n=\lfloor rn
	\rfloor $.  Let 
	$ S\subs [0,1]$ be a countable dense subset. 
	For each $n\geq n_0$,
	define $ \ga_n\colon S\to X $ by $ \ga_n(s)=p^n_{\lfloor N_ns\rfloor} $ 
	$(s\in S)$. Then $ \ga_n $ is discontinuous, but 
	\begin{equation}\label{E:almost}
	 	d(\ga_n(s),\ga_n(s'))\leq d_0(\ga_n(s),\ga_n(s'))\leq 
		\frac{1}{n}\big\vert\lfloor N_ns\rfloor - 
			\lfloor N_ns'\rfloor \big\vert \leq r\abs{s-s'}+\frac{2}{n}\quad
			(s,\,s'\in S).
	\end{equation}
	The triangle inequality for $d_0$ shows that $p^n_k$ lies in the 
	compact set  $ \ol{B_{d_0}(p;r)}$ for all $k\in [N_n]$.  Using
	Cantor's diagonal argument one can obtain a subsequence $(\ga_{n_\nu})$ with
	the property that for all $s\in S$, $\ga_{n_\nu}(s)$ $ d $-converges to some
	$\ga(s)\in X$ as $\nu\to \infty$. Then, by \eqref{E:almost}, 
	\begin{equation*}\label{E:Lipschitz}
		d(\ga(s),\ga(s'))\leq r\abs{s-s'}\text{ for all $s,s'\in S$.}
	\end{equation*}
	Since $(X,d)$ is complete, $ \ga $ can be continuously extended to $ [0,1]
	$, so that 
	\begin{equation*}
		d(\ga(t),\ga(t'))\leq r\abs{t-t'}\text{ for all $t,t'\in [0,1]$.}
	\end{equation*}
	This implies that $ L(\ga) \leq d_0(p,q)=r $. 
	But $ d_0\leq \bar{d} $ and $ \bar{d}(p,q)\leq L(\ga)
	$ by the definition of $ \bar{d} $, hence
	\begin{equation*}
		d_0(p,q)=\bar{d}(p,q)=L(\ga).
	\end{equation*}
	Thus $(X,\bar{d})=(X,d_0)$ is a geodesic metric space. 

	Now assume that (ii) is satisfied. 
	Since $ d_0\leq (d_0)_\eps\leq (d_0)_0 $ always holds, we deduce that 
	$ (d_0)_\eps=d_0 $ for all $ \eps>0 $. Therefore, given $
	p,q\in X$ with $ d_0(p,q)<\infty $ and $ \eps>0 $, 
	one can find an $ \eps $-chain $ c $ (with
	respect to $ d_0 $) joining $ p $ to $ q $ for which $ L_{d_0}(c)\leq
	d_0(p,q)+\eps $. Hence $ (X,d_0) $ admits ``approximate midpoints''
	(cf.~\cite{BriHae}, p.~32 or \cite{BurBurIva}, p.~42). Since $ (X,d_0) $ is
	complete by \lref{L:completeness}, it must be a 
	a length space, i.e., $ \overline{d_0} = d_0 $. 
	On the other hand, \eqref{E:ineqs} and \lref{R:ineqs2} imply that 
	\begin{equation*}
		d\leq d_0\leq \bar{d}\leq \overline{d_0}.
	\end{equation*}
	Therefore $ d_0=\bar{d} $.

	Finally, if condition (iii) holds then it is obvious from
	\eqref{E:ineqs} that $ d_0 = \bar{d} $.
\end{proof}

\section{Examples}\label{S:examples}

It is erroneously asserted in Exercise 3.1.26 of \cite{BurBurIva} that $ d_0 $
coincides with $ \bar{d} $ whenever $ (X,d) $ is complete.\footnote{The
definition of $ d_\eps $ (and hence that of $ d_0 $) in \cite{BurBurIva} is
different from the one we have given, but the two are easily proved to be
equivalent.}  
However, as the following examples show, none of the hypotheses in (i) and (ii)
of \tref{T:main} can be omitted. 
\begin{exm}[cf.~Figure \ref{F:Y}]\label{E:discrete}
	For each integer $ k \geq 2 $, let
	\begin{alignat*}{9}\label{E:S_k}
		S_k=\,&\set{(1,0,\dots,0,t,0,\dots)\in \ell^\infty}{t\in [0,1]}\,\cup \\
		\notag
		\cup\,&\set{(t,0,\dots,0,1,0,\dots)\in \ell^\infty}{t\in [\tfrac{1}{k},1]}
		\,\cup \\ \notag
		\cup\,&\set{(\tfrac{1}{k},0,\dots,0,t,0,\dots)\in \ell^\infty}{t\in [0,1]},
	\end{alignat*}
	where in each description the only nonzero coordinates are the first and
	$k$-th. Let 
	\begin{equation*}
		p=(0,0,\dots), \ \ q=(1,0,0,\dots) 
	\end{equation*}
	and let $ Y=\se{p}\cup\bcup_{k=2}^{\infty}S_k $, equipped with the
	metric $ d $ derived from $ \norm{\cdot}_\infty $. 
	Then $ (Y,d) $ is
	complete, $ \sig $-compact (hence separable)
	and connected (though not path-connected).
	It is not locally compact, since no neighborhood of $ p $ or $ q $ is
	precompact.
	It will be proved shortly that $ d_0(p,y) = 3 + \bar{d}(q,y) $ for all $
	y\neq p $. In particular, 
	\begin{equation*}
		d(p,q)=1<d_0(p,q)=3<\bar{d}(p,q)=\infty,
	\end{equation*}
	and $ q $ is simultaneously the point $ d_0 $-closest to $ p $ and 
	one of the points maximizing the $ d $-distance to $ p $. Note that $
	(Y,d_0) $ is disconnected even though $ (Y,d) $ is connected, and that the $
	d_0$-topology is strictly finer than the $ d $-topology, even
	though $ d_0 $ only takes on finite values. 
\begin{figure}[ht]
	\begin{center}
		\includegraphics[scale=.79]{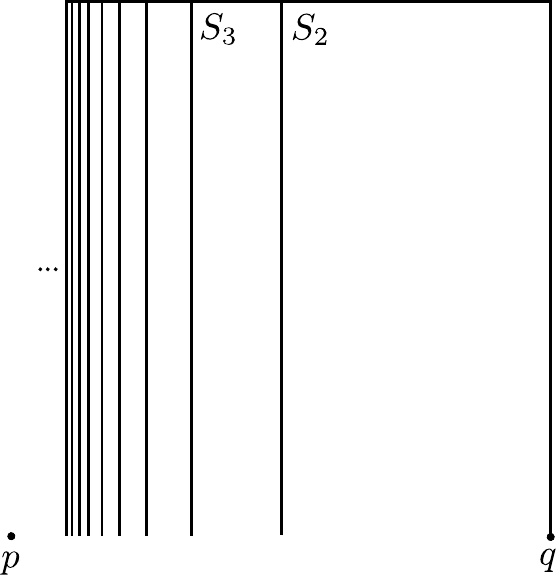}
		\caption{A two-dimensional representation of the space $ Y $ described in  
			\eref{E:incompact}. The segments are drawn to scale,
			but the figure is highly distorted since 
		the $ S_k $ lie in mutually distinct planes of $ \ell^\infty $ and $
		S_k\cap S_l=\se{q} $ for $ k\neq l $.} 
		\label{F:Y}
	\end{center}
\end{figure}
\end{exm}

In the sequel we say that an $ \eps $-chain $ c $ can be \tdef{reduced}
to another $ \eps $-chain $ c' $ if they join the same pair of points and
$ L(c')\leq L(c) $.

\begin{lem}\label{L:Y}
	Let $ (Y,d) $ be as described in \eref{E:discrete}. Then $ 
	d_0(p,y)=3+\bar{d}(q,y) $ for any $ y\neq p$, and the restrictions 
	of $ d_0 $ and $ \bar{d}  $ to $ Y\ssm\se{p} $ coincide. 
	Consequently, $d<d_0<(d_0)_0=\bar{d} $.
\end{lem}
\begin{proof}
	Suppose that $ x\in S_m $, $y\in S_l $. Take $
	\eps<\min\{\frac{1}{m(m+1)},\frac{1}{l(l+1)}\} $. Using that 
	\begin{equation*}
		B_d(S_m\cup S_l;\eps)=S_m\cup S_l\cup B_d(q;\eps), 	
	\end{equation*}
 	it is not hard to see that any 
	$ \eps $-chain joining $ x $ to $ y $ can be reduced to a chain 
	contained in $ S_m\cup S_l $. Any such chain has length at least $
	\bar{d}(x,y)-4\eps $, as one verifies directly. Thus $
	d_0(x,y)\geq \bar{d}(x,y) $, and the reverse inequality is guaranteed by 
	\eqref{E:ineqs}. 

	Now let $ y\in S_l $ and $ \eps<\frac{1}{l(l+1)} $.  
	Since $ B_d(S_l;\eps) =S_l\cup B_d(q;\eps) $, any $ \eps $-chain 
	joining $ y $ to $ p $ can be reduced to the concatenation of chains 
	$ c_1 $ and $ c_2 $ joining $ y $ to $ q $ and $ q $ to $ p $, respectively. 
	Let $ x = (\frac{1}{m},0,\dots)	$, where $ m $ is the largest integer
	satisfying $ \frac{1}{m} \leq \eps $. Adjoining $ x $ to $ c_2 $, 
	an $ \eps $-chain $ c $ joining $ q $ to $ x $ satisfying $
	L(c)\leq L(c_2)+\eps $ is obtained. But direct verification shows that 
	$ L(c) \geq \bar{d}(q,x)-2\eps	$. Thus
	\begin{equation*}
		L(c_2)\geq 3-3\eps, \text{ \,while\, } L(c_1) \geq \bar{d}(q,y)
		- 4\eps
	\end{equation*} 	
	as in the preceding paragraph. 
	Hence $ d_0(p,y)\geq 3+\bar{d}(q,y) $. 
	On the other hand, given $ \eps > 0 $, an $ \eps $-chain $ (x_0,\dots,x_n) $
	joining $ p $ to $ y $ of length at most $ 3+\bar{d}(q,y) $ can easily be
	constructed: take $ x_1=(\frac{1}{m},0,\dots) $ with $m>\eps^{-1} $ and
	choose the remaining $ x_k $ monotonically along the geodesic joining 
	$ x_1 $ to $ y $.  Thus $ d_0(p,y)=3+\bar{d}(q,y) $. This completes the
	proof of the asserted description of $ d_0 $. 
	
	Since $ d_0(p,y)\geq 3 $ for all $ y\neq p $, $
	(d_0)_0(p,y)=\infty=\bar{d}(p,y) $ for any such $ y $.  In addition, $
	(d_0)_0(x,y)=\bar{d}(x,y) $ for $ x,y\neq p $ since $ d_0(x,y)=\bar{d}(x,y)
	$ and $ d_0\leq (d_0)_0\leq \bar{d} $ always holds.  Therefore $
	d<d_0<(d_0)_0=\bar{d} $ as claimed.
\end{proof}

It is natural to think that the failure of $ d_0 $ to agree with $
\bar{d} $ in the previous example has more to do with the fact that $ (Y,d) $ is
not path-connected than with absence of local compactness; after all, $
\bar{d} $ is defined in terms of lengths of paths.  However, this is not the
case. 

\begin{exm}\label{E:incompact}
	For each integer $ k\geq 2 $, let $
	f_k\colon [\frac{1}{k+1},\frac{1}{k}] \to \R $ be a smooth function
	such that	
	\begin{equation*}
		f_k\big(\tfrac{1}{k+1}\big)=\tfrac{1}{k+1},\ \ 
		f_k\big(\tfrac{1}{k}\big)=0\ \  \text{and} \ \ \abs{f_k(t)}\leq t
		\text{\ \ for all $ t\in \big[\tfrac{1}{k+1},\tfrac{1}{k}\big] $}.
	\end{equation*}
	Let $ \eta_k\colon \big[\frac{1}{k+1},\frac{1}{k}\big] \to \ell^\infty$ be
	given by $ \eta_k(t)=(t,0,\dots,0, f_k(t),0,\dots) $, where $ f_k $
	appears in the $ k $-th coordinate. Finally, let $ \ga_k\colon
	\big[0,\frac{1}{k}\big]\to \ell^\infty $ be the concatenation of the line
	segment joining $ p=(0,0,\dots) $ to $ \eta_k\big(\frac{1}{k+1}\big) $ and $
	\eta_k $. Then $ \ga_k $ joins $ p $ to $ \big(\frac{1}{k},0,\dots) $ 
	without passing through $ \big( \frac{1}{m},0,\dots \big) $ for any $
	m\neq k $, and (because $ \abs{f_k(t)}\leq t $) the intersection of its
	image with any $ d $-ball centered at $ p $ is an arc of $ \ga_k $. 
	Let $ f_k $ be chosen so 
	as to have $ L(\ga_k) = 3+\frac{1}{k} $.
	
	Now define $ X $ to be the union of the set $ Y $ of
	\eref{E:discrete} and the images of all paths $ \ga_k $ $ (k\geq 2) $,
	equipped with the restriction $ d $ of the metric of $ \ell^\infty $. Then $
	(X,d) $ is not locally compact, but it is complete, $ \sig $-compact,
	and any pair of points in it can be joined by a rectifiable path. Moreover,
	given $ x\in X $ and a $ d $-neighborhood $ U $ of $ x $, there exists a $
	d $-ball $ B\subs U$ containing $ x $ such that any two points of $ B $ can
	be joined by a rectifiable path contained in $ B $.  Still, 
	\begin{equation*}
		d(p,q)=1<d_0(p,q)=3 <\bar{d}(p,q)=6.
	\end{equation*}
	A minimizing geodesic connecting $ p $ and $ q $ is the 
	concatenation of $ \ga_k $ and $ S_k $, for any $ k$. 
\end{exm}
	
\begin{exm}\label{E:incomplete}
	Let $ X = \R^2\ssm ( \se{0}\times [-1,+1] ) $, furnished with the
	restriction $ d $ of the Euclidean metric on $ \R^2 $. Then $ d_0 = d $, but
	if $ p=(-1,0) $, $ q = (1,0) $, then $ d_0(p,q)=2 < 2\sqrt{2}=\bar{d}(p,q)$.
	The space $ (X,d) $ is locally compact and (locally) path-connected 
	through rectifiable paths, but
	not complete.
\end{exm}

\begin{exm}\label{E:rational}
	Let $ X $ be the set of all rational numbers, equipped with the restriction
	of the Euclidean metric $ d $. Then $ (d_0)_0=d_0=d $ does not coincide with
	$ \bar{d} $.
\end{exm}

\section{Iterates of $ d_0 $}\label{S:iterates}

One of the fundamental properties of the induced length metric $ \bar{d} $ is
that $ \bar{\bar{d}} = \bar{d} $ . In contrast, $ d_0 $ may not coincide
with $ (d_0)_0 $, as shown by the space $ (Y,d) $ of \eref{E:discrete}.
For a metric $ d $ and $ n\in \N^+ $, let 
\begin{equation*}
	d_0^{0} = d \text{\ \,and\,\ }d^{n}_0 = (d_0^{n-1})_0. 
\end{equation*}  It follows from \eqref{E:ineqs} and \cref{C:mixed} (applied to
$ (X,\bar{d})) $ that 
\begin{equation*}\label{E:ineqs2}
	d\leq d_0^{n-1} \leq d_0^{n}\leq \bar{d} \text{\ \ for any $ n \in \N^+ $}.
\end{equation*}

\begin{prop}\label{P:stabilizer}
	Let $ (X,d) $ be a complete metric space and suppose that 
	$ d_0^n=d_0^{n+1} $ for some $ n\in \N $. Then $ d_0^m=\bar{d} $ for all $
	m\geq n $.
\end{prop}
\begin{proof}
	If $ d_0 = d $ then certainly $ (d_0)_0=d_0 $, hence
	it may be assumed than $ n\geq 1 $. As a consequence of \lref{R:ineqs2} and the
	preceding inequalities, 
	\begin{equation*}
		\bar{d}\leq \overline{d_0^{n-1}}.  \end{equation*}
	The metric $ d_0^{n-1} $ is complete by repeated use of 
	\lref{L:completeness}. By 
	case (ii) of \tref{T:main} applied to $ d_0^{n-1} $,
\begin{equation*}
	d_0^{n} = (d_0^{n-1})_0 = \overline{d_0^{n-1}}. 
\end{equation*}  
	Therefore $ \bar{d}=d_0^n=d_0^{n+1}=\dots $.
\end{proof}
\begin{urem}
	The assumption that $ (X,d) $ is complete cannot be omitted, as shown by
	\eref{E:rational}. 
\end{urem}
We shall now describe complete connected spaces $ (Y_{n},d) $ for which 
\begin{equation*}
	d<d_0<\dots<d_0^{n}<d_0^{n+1}=\bar{d},
\end{equation*}
and a complete connected space $ (Y_{\infty},d) $ for which 
$ \lim_{n\to \infty}d_0^{n}<\bar{d}
$. 

\begin{exm}\label{E:iterated}
	Set $ (Y_{1}, d) =(Y,d) $, where the latter was described in
	\eref{E:discrete}. 
	
	For each $ m\in \N^+ $, take the disjoint union of $
	m $ copies $ Z^j_m $ ($ j\in [m] $) of $ Y_{1} $ with its metric $ d $ 
	contracted by the factor $ \frac{1}{m} $. Let $ Y_2 $ be the result of gluing:
	\begin{enumerate}
		\item [(a)] The point $ q $ of $ Z_m^j $ to the point $ p $ of 
			$ Z_m^{j+1} $ for each $ j\in [m-1] $ and $ m\in \N^+ $.
		\item [(b)] The points $ p $ of $ Z^1_m $ for all $ m\in \N^+ $ to a 
			single point, still denoted $ p $.
		\item [(c)] The points $ q $ of $ Z^m_m $ for all $ m\in \N^+ $ to a 
			single point, still denoted $ q $.
	\end{enumerate}
	
	The spaces $ Y_{n} $ for integer $ n\geq 2 $ are defined inductively by
	replacing $ Y_1 $ and $ Y_2 $ by $ Y_{n-1} $ and $ Y_n $ in the preceding
	paragraph, respectively.  Note that $ Y_n $ contains an isometric copy of $
	Y_{n-1} $, namely, $ Z^1_1 $ (or more precisely its image under the gluing),
	and the points $ p $ and $ q $ of $ Y_{n-1} $ are thereby identified with
	the corresponding points $ p $ and $ q $ of $ Y_{n} $. 
	Finally, let $
	Y_{\infty} = \bcup_{n=1}^\infty Y_n $, where $ Y_{n-1} $ is regarded as a
	subspace of $ Y_{n} $ as was just indicated. The metric in $ Y_\infty $ is
	uniquely determined since any two points in it lie in the same $ Y_n $
	whenever $ n $ is sufficiently large. 
	
	Straightforward inductive arguments show that $ Y_n $ is connected and
	complete for every $ n $; $ Y_\infty $ is connected as the union
	of an increasing family of connected spaces, and it is complete because a
	Cauchy sequence either eventually lies in some $ Y_n $, or converges to $ p
	$ or $ q $.
\end{exm}

\begin{prop}\label{P:ascending}
	Let $ (Y_n,d) $ be as in \eref{E:iterated}. Then
	$ d^n_0(p,y)=3+\bar{d}(q,y) $ for all $ y\neq p$ and the
	restrictions of $ d_0^n $ and $ \bar{d} $ to $ Y_n\ssm\se{p}$  
	agree. In particular, 
	\begin{equation}\label{E:last}
		d<d_0<\dots<d_0^{n-1}<d_0^n<d_0^{n+1}=\bar{d} \text{ in $ Y_n $}.
	\end{equation}
	In $ (Y_\infty,d) $ we have that  $ d_0^n<d_0^{n+1} $ for all 
	$ n\in \N $ and $ \lim_{n\to\infty} d_0^n < \bar{d} $.
\end{prop}
\begin{proof}
	By induction on $ n $. For $ n = 1 $ the assertions were proved in
	\lref{L:Y}. Assume that the conclusion holds for $ Y_{n-1} $ with $ n\geq
	2$. From now on $ Z^j_m $ denotes the image of this space under the gluing
	which defines $ Y_n $. Recall that $ Z^j_m$ is isometric to $
	(Y_{n-1},\frac{1}{m}d_{Y_{n-1}}) $ for all $ j\in[m] $, and note that the
	restriction of $ \bar{d} $ to $ Z^j_m $ and the length metric on the latter
	induced by the restriction of $ d $ coincide.
	
	Let $ j<m $, $ p\neq x\in Z^j_m $ and $ y\in Y_n $
	be arbitrary. 
	If $ y'\nin Z^j_m $, then $ d_0^{n-1}(x,y') \geq \frac{3}{m}$ (here the
	hypotheses that $ j<m $ and $ x\neq p $ are essential). Hence
	\begin{equation*}
		d_0^{n}(x,y)=\bar{d}(x,y)=\infty \text{\ \,for all $ y\nin Z^j_m $}.
	\end{equation*}
 	If $ y\in Z_m^j $, then $ d_0^{n}(x,y)=\bar{d}(x,y) $ by the induction 
	hypothesis applied to $ Z^j_m $. In particular, taking $ y = p $, 
	\begin{equation}\label{E:special}
		d_0^{n}(p,x) = \bar{d}(p,x) = \infty = 3 + \bar{d}(q,p) \text{ for $
			p\neq x\in Z^j_m $ 
		with $ j<m $.}
	\end{equation}
	
	Now let $ x,y\in Z^m_m \ssm Z_m^{m-1}	$, where we set $ Z_1^0=\se{p} $ for
	convenience, and let $ \eps<\frac{3}{m} $.  Let $ c=(x=x_0,\dots,x_N=y) $ be
	an $ \eps $-chain with respect to $ d_0^{n-1} $ and assume that $ c $ is not
	contained in $ Z^m_m $.  Let $ k_0,k_1 \in [N]$ be the greatest
	(resp.~smallest) indices such that $ x_0,\dots,x_{k_0} \in Z^m_m$ and $
	x_{k_1},\dots,x_N\in Z^m_m $.  Then $
	d_0^{n-1}(x_{k_0},q),\,d_0^{n-1}(x_{k_1},q)\leq \eps $ by the choice of $
	\eps $ and $ k_i $ (informally, the chain cannot leave $ Z^m_m $ through $
	Z^{m-1}_m $ since $ \eps<\frac{3}{m} $).  Replacing $
	(x_{k_0},\dots,x_{k_1}) $ by $ (x_{k_0},q,x_{k_1}) $, we obtain an $ \eps
	$-chain $ c' $ contained in $ Z^m_m $ satisfying $ L(c')\leq L(c)+2\eps $.
	This shows that the restriction of $ d_0^n $ to $ Z^m_m $ agrees with $
	(d|_{Z^m_m\times Z_m^m})_0^n $, that is, with the original metric $ d_0^n $
	on $ Z^m_m $. Hence, by the induction hypothesis applied to $ Z^m_m $,
	\begin{equation*}
		d_0^n(x,y)=\bar{d}(x,y)\text{ for $ x,y\in Z^m_m\ssm Z^{m-1}_m$},~m\in
		\N^+.
	\end{equation*}
	Now let $ x\in Z^m_m\ssm Z_m^{m-1} $ and $ y\in Z^l_l \ssm Z_l^{l-1} $ 
	with $ m\neq l $, and take $ \eps<\min\se{\frac{3}{m},\frac{3}{l}} $. A
	straightforward modification of the preceding argument can be used to
	compare an $ \eps $-chain joining $ x $ to $ y $ to the concatenation of $
	\eps $-chains joining $ x $ to $ q $ and $ q $ to $ y $, allowing one to
	conclude that 
	\begin{equation*}
		d_0^n(x,y)=\bar{d}(x,q)+\bar{d}(q,y)=\bar{d}(x,y).
	\end{equation*}

	It remains to compute $ d_0^{n}(p,y) $ for $ y\in Z^l_l\ssm Z^{l-1}_l $.
	Using the argument of the preceding paragraph again, one
	deduces that 
	\begin{equation*}
		d_0^n(p,y)=d_0^n(p,q)+d_0^n(q,y)=d_0^n(p,q)+\bar{d}(q,y).
	\end{equation*}
	Any $ \eps $-chain joining $ p $ to $ q $ can be reduced to a chain 
	$ c' $ contained in $ \bcup_{j=1}^mZ^j_m $ for some $ m > \eps^{-1}
	$. By the definition of gluing
	and the inductive hypothesis, $ d_0^{n-1}(a,b)\geq \frac{3}{m}\abs{i-j} $ 
	whenever $ a\in Z_m^i $, $ b\in Z_m^j $ ($ a\neq b $). Therefore $ L(c')\geq
	3-\frac{1}{m}\geq 3-\eps$, so that $ d_0^n(p,q)\geq 3 $.
	On the other hand, it is easy to construct explicit $ \eps $-chains of
	length 3 joining $ p $ to $ q $ for any $ \eps>0 $: take $ m>\eps^{-1} $ and
	consider $ (p,q_1,\dots,q_{m-1},q)  $, where $
	q_j\in Z^j_m $ denotes the point glued to $ Z^{j+1}_m $ in step (a) of
	\eref{E:iterated}. Thus $ d_0^n(p,y)=3+\bar{d}(q,y) $ if $ y\in Z_m^m\ssm
	Z_m^{m-1}
	$ for some $ m $, 
	and for other $ y\neq p $ this equality holds by \eqref{E:special}.
	
	This completes the description of the metric $ d_0^n $ on $ Y_n $. 
	Since $ d_0^n(p,y)\geq 3 $ for
	all $ y\in Y_n\ssm\se{p} $, 
	\begin{equation*}
		d_0^{n+1}(p,y)= \infty = \bar{d}(p,y) \text{\ \,for all such $ y $}.	
	\end{equation*} 
	For $ x,y\neq p$, certainly $ d_0^{n+1}(x,y)=\bar{d}(x,y) $, since  $
	d_0^n\leq d_0^{n+1}\leq \bar{d} $ always holds and $ d_0^n(x,y)=\bar{d}(x,y)
	$ already.  Hence $ d_0^{n+1}=\bar{d} $ in $ Y_n $.
	The inequalities in \eqref{E:last} are a consequence of \pref{P:stabilizer}
	and the fact that $ d_0^n<\bar{d} $.

	Finally, since $ d_0^n(p,q)=3 $ in $ Y_n $ and $ Y_n $
	is isometrically embedded in $ Y_\infty $ for all $ n\in \N^+ $,
	\begin{equation*}
		\lim_{n\to \infty} d_0^n(p,q)\leq 3 <\bar{d}(p,q)=\infty\text{\ \,in
			$Y_{\infty}$}.
	\end{equation*}
	Thus $ \lim_nd_0^n<\bar{d} $. Since $
	(Y_\infty,d) $ is complete, \pref{P:stabilizer} guarantees that $
	d_0^n<d_0^{n+1} $ for all $ n\in \N $. 
\end{proof}
\begin{urem}
	By replacing $ Y $ by $ X $ throughout in \eref{E:iterated}, where $ (X,d) $
	is the space described in \eref{E:incompact}, one obtains a family $ (X_n,d)
	$ of spaces satisfying the same inequalities as in \eqref{E:last}, but with
	the additional property that each $ (X_n,d) $ is (locally) path-connected
	through rectifiable paths. To prove this one can use the argument given
	above, with small variations; however, the proof is more cumbersome since
	the metric $ d_0 $ on $ X_n $ does not have such a simple description as
	that on $ Y_n $.
\end{urem}


\subsection*{Acknowledgements}
The author thanks C.~Gorodski and F.~Gozzi for valuable discussions on this
and related topics. 
Financial support from \tsc{fapesp} is gratefully acknowledged.


\providecommand{\bysame}{\leavevmode\hbox to3em{\hrulefill}\thinspace}
\providecommand{\MR}{\relax\ifhmode\unskip\space\fi MR }
\providecommand{\MRhref}[2]{%
  \href{http://www.ams.org/mathscinet-getitem?mr=#1}{#2}
}
\providecommand{\href}[2]{#2}


\begin{thebibliography}{1}

\bibitem{BriHae}
M.~Bridson and A.~Haefliger, \emph{Metric spaces of non-positive curvature},
  Grundlehren der mathematischen Wissenschaften, vol. 319, Springer Verlag,
  1999.

\bibitem{BurBurIva}
D.~Burago, Y.~Burago, and S.~Ivanov, \emph{A course in metric geometry},
  Graduate Studies in Mathematics, vol.~33, American Mathematical Society,
  2001.

\end{thebibliography}

\vspace{12pt}

\noindent{\small \tsc{Instituto de Matem\'atica e Estat\'istica,
	Universidade de S\~ao Paulo (IME-USP) \\ Rua~do Mat\~ao 1010, Cidade
	Universit\'aria -- S\~ao Paulo, SP
05508-090, Brazil}}\\
\noindent{\ttt{pedroz@ime.usp.br}}

\end{document}